\documentclass[12pt,a4paper]{amsart}

\usepackage{amssymb,latexsym,graphicx}
\usepackage{color}
\usepackage{colortbl}
\usepackage{subcaption, float}
\usepackage[cp1252]{inputenc} 
\usepackage[T1]{fontenc}
\usepackage{lmodern}
\usepackage[english]{babel}
\usepackage{enumerate}
\usepackage[section]{placeins}
\usepackage[pagebackref=true]{hyperref}
\hypersetup{pdftex,colorlinks=true,linkcolor=blue,citecolor=blue,urlcolor=blue}

\usepackage{currfile}

\usepackage[margin=3cm]{geometry}

\usepackage{version}

\theoremstyle{plain} 
\newtheorem{theorem}{Theorem}[section]
\newtheorem{lemma}[theorem]{Lemma}
\newtheorem{proposition}[theorem]{Proposition}

\newtheorem{claim}[theorem]{Claim}

\theoremstyle{definition}
\newtheorem{definition}[theorem]{Definition}
\newtheorem{example}[theorem]{Example}
\newtheorem{remark}[theorem]{Remark}
\newtheorem{remarks}[theorem]{Remarks}

\numberwithin{equation}{section}
\numberwithin{figure}{section}
\numberwithin{table}{section}

\definecolor{purple}{RGB}{127,0,255}

\newcommand{\clb}{\color{blue}}
\newcommand{\clr}{\color{red}}

\newcommand{\clm}{\color{magenta}}

\definecolor{lgray}{gray}{0.90}

\newcommand{\noid}{\noindent $\diamond$~}

\newcommand{\K}{\mathbb{K}}

\newcommand{\R}{\mathbb{R}}
\newcommand{\bS}{\mathbb{S}}

\newcommand{\cD}{\mathcal{D}}

\newcommand{\cS}{\mathcal{S}}

\newcommand{\cZ}{\mathcal{Z}}

\newcommand{\mf}{\mathfrak}

\newcommand{\sm}{\!\setminus\!}

\newcommand{\set}[1]{\left\lbrace #1 \right\rbrace}
\newcommand{\pib}{\frac{\pi}{2}}
\newcommand{\pic}{\frac{\pi}{3}}

\newcommand{\pid}{\frac{\pi}{4}}

\newcommand{\pf}{\emph{Proof.~}}

\newcommand{\rP}{\mathbb{P}^2}
\newcommand{\gammat}{\widetilde{\gamma}}
\newcommand{\Gammat}{\widetilde{\Gamma}}
\newcommand{\bD}{\partial \cD}

\title[Euler-type formula M\"{o}bius]{An Euler-type formula for partitions\\ of the M\"{o}bius strip}

\date{\today}

\author[P. B\'{e}rard]{Pierre B\'erard}
\author[B. Helffer]{Bernard Helffer}

\address{PB: Universit\'{e} Grenoble Alpes and CNRS\\
Institut Fourier, CS 40700\\ 38058 Grenoble cedex 9, France.}
\email{pierrehberard@gmail.com}

\address{BH: Laboratoire Jean Leray, Universit\'{e} de Nantes and CNRS\\
F44322 Nantes Cedex, France and LMO (Universit\'e Paris-Sud).}
\email{Bernard.Helffer@univ-nantes.fr}

\keywords{Spectral theory,  M\"{o}bius strip, Laplacian, Partitions.}

\subjclass[2010]{58C40, 49Q10.}

\date{\today~(\currfilename)}

\excludeversion{cylinder}

\excludeversion{projective}

\excludeversion{append-A}

\excludeversion{rem}

\begin{document}

\begin{abstract}
The purpose of this note is to prove an Euler-type formula for partitions of the M\"{o}bius strip. This formula was introduced in our joint paper with R.~Kiwan, ``Courant-sharp property for Dirichlet eigenfunctions  on the M\"{o}bius strip'' (arXiv:2005.01175).
\end{abstract}

\maketitle

\vspace{2cm}

\section{Introduction}

In \cite{BHK}, in collaboration with R.~Kiwan, we investigated the Courant-sharp property for the eigenva\-lues of the Dirichlet Laplacian on the square M\"{o}bius strip, equipped with  the flat metric. We pointed out that the orientability of the nodal domains (more precisely the fact that they are all orientable or not) can be detected by an Euler-type formula. The purpose of the present note is to establish this formula in the framework of partitions.\medskip

In Section~\ref{S-parti}, we fix the notation and recall (or modify) some of the definitions for partitions given in \cite{BH0}, and we state the Euler-type formula we are interested in, see Theorem~\ref{T-euler-nor}. The proof of the theorem is given in Section~\ref{S-pfnor}.

\section{Partitions}\label{S-parti}

\subsection{Definitions and notation}

Let $\Sigma$ denote a compact Riemannian surface with or without boundary. We consider Euler-type formulas in the general framework of partitions. We first recall (or modify) some of the definitions given in \cite{BH0}.\medskip

A \emph{$k$-partition} of $\Sigma$ is a collection,  $\cD = \{D_j\}_{j=1}^{k}$, of $k$ pairwise disjoint, connected, open subsets of $\Sigma$. We furthermore assume that the $D_j$'s are piecewise $C^1$, and that
%
\begin{equation}\label{E-parti-str}
\mathrm{Int}(\overline{\cup_j D_j} ) =  \Sigma\,.
\end{equation}

The \emph{boundary set} $\partial \mathcal{D}$ of a partition
$\mathcal{D} = \{D_j\}_{j=1}^{k} \in \mathfrak{D}_k(\Sigma)$ is the
closed set,
\begin{equation}
\partial \mathcal{D} = \overline{\cup_j(\partial D_j \cap \Sigma)}\,.
\end{equation}

\begin{definition}\label{D-parti-ess}
A partition $\cD = \set{D_j}_{j=1}^k$ is called \emph{essential} if, for all $j, \,1 \le j \le k$,
\begin{equation}\label{E-parti-ess}
\mathrm{Int}(\overline{D_j} ) = D_j\,.
\end{equation}
\end{definition}%

\begin{definition}\label{D-parti-reg}
A \emph{regular $k$-partition} is a $k$-partition
whose boundary set $\bD$ satisfies the following properties:
\begin{itemize}
    \item[(i)] The boundary set $\bD$ is locally a piecewise $C^1$  immersed curve in $\Sigma$,
    except possibly at finitely many points $\{y_i \in \bD \cap
    \Sigma\}$ in a neighborhood of which $\bD$ is the union of
    $\nu(y_i)$ $C^1$ semi-arcs meeting at $y_i, \nu(y_i) \ge 3$.
    \item[(ii)] The set $\bD \cap \partial \Sigma$ consists of finitely
    many points $\{z_j\}$. Near the point $z_j$, the set $\bD$ is the
    union of $\rho(z_j)\ge 1$ $C^1$ semi-arcs hitting $\partial \Sigma$ at $z_j$.
    \item[(iii)] The boundary set $\bD$ has the following \emph{transversality property}:
    at any \emph{interior singular point} $y_i$, the semi-arcs meet transversally; at any \emph{boundary singular point} $z_j$, the semi-arcs meet transversally, and they meet  the boundary $\partial \Sigma$ transversally.
\end{itemize}
\end{definition}%

The subset of regular $k$-partitions is denoted by $\mathcal{R}_k(\Sigma) \subset \mf{D}_k(\Sigma)$. When $\cD$ is a regular partition, we denote by  $\cS(\cD)$ the set of singular points of $\bD$.\medskip

\begin{definition}\label{D-parti-nor}
A regular $k$-partition $\cD = \set{D_j}_{j=1}^k$ is called \emph{normal}, if it satisfies the additional condition, for all $j, ~1 \le j \le k$,
\begin{equation}\label{E-parti-nor}
\forall x \in \partial D_j, ~\exists r >0 \text{~~s. t.~~} B(x,r)\cap D_j \text{~is~} connected.
\end{equation}
\end{definition}%

\begin{remark}\label{R-parti-nor}
The definition of a normal partition implies that each domain in the partition is a topological manifold with boundary (actually a piecewise $C^1$ surface with boundary, possibly with corners).  A normal partition is essential.
\end{remark}%

\begin{example}\label{E-parti-nod}
The partition of a compact surface $\Sigma$ (with or without boundary) associated with an eigenfunction is called a \emph{nodal partition}. The domains of the partition are the nodal domains of $\Phi$, the boundary set $\bD$ is the nodal set $\cZ(\Phi)$, the singular set $\cS(\cD)$ is the set of critical zeros of $\Phi$. This is an example of an essential, regular partition. Nodal partitions are not necessarily normal due to the singular set, see Figure~\ref{F-em-1} (middle) and Figure~\ref{F-em-3}\,(A).
\end{example}%

For a partition $\cD \in \mf{D}(\Sigma)$, we introduce the following numbers.
\begin{itemize}
  \item[(a)] $\kappa(\cD,\Sigma)$ denotes the number of domains of the partition;
  \item[(b)] $\beta(\cD,\Sigma)$ is defined as $b_0(\bD \cup \partial \Sigma) - b_0(\partial \Sigma)$, the difference between the number of connected components of $\bD \cup \partial \Sigma$, and the number of connected components of $\partial \Sigma$;
  \item[(c)] $\omega(\cD,\Sigma)$ is the \emph{orientability character} of the partition,
\begin{equation*}
\left\{
\begin{array}{ll}
\omega(\cD,\Sigma) = 0, & \text{if all the domains of the partition are orientable};\\[5pt]
\omega(\cD,\Sigma) = 1, & \text{if at least one domain of the partition is non-orientable.}
\end{array}%
\right.
\end{equation*}

\end{itemize}

Obviously, $\omega(\cD,\Sigma) = 0$ whenever the surface $\Sigma$ is orientable.\medskip

\begin{remarks}\label{R-parti-or}~
\begin{enumerate}
  \item[(i)] We use the definition of orientability given in \cite[Chap.~5.3]{BeGo1988} (via differential forms of degree $2$), or the similar form given in \cite[Chap.~4.5]{GaXu} in the setting of topological manifolds (via the degree).  A topological surface is orientable if one can choose an atlas whose changes of chart are homeomorphisms with degree $1$.
  \item[(ii)] A compact surface (with boundary) is non-orientable if and only if it contains the homeomorphic image of a M\"{o}bius strip. One direction is clear since the M\"{o}bius strip is not orientable. For the other direction, one can use the classification of compact surfaces (with boundary).
\end{enumerate}
\end{remarks}%

For a regular partition $\cD$, we define the \emph{index} of a point $x
\in \mathcal{S}(\mathcal{D})$ to be,
\begin{equation}
\iota (x) = \left\{
\begin{array}{ll}
\nu(x) - 2\,, & \text{~if~} x \text{ ~is an interior singular point},\\
\rho(x)\,,  & \text{~if~} x \text{ ~is a boundary singular point}.
\end{array}
\right.
\end{equation}
For a regular partition, define the number $\sigma(\mathcal{D},\Sigma)$ to be,
\begin{equation}
\sigma(\mathcal{D},\Sigma) = \frac{1}{2}\, \sum_{x \in \mathcal{S}(\mathcal{D})}
\iota(x)\,.
\end{equation}

Finally, we introduce the number
\begin{equation}\label{E-parti-del}
\delta(\cD,\Sigma) = \omega(\cD,\Sigma) + \beta(\cD,\Sigma) + \sigma(\cD,\Sigma) - \kappa(\cD,\Sigma)\,.
\end{equation}

\begin{lemma}[Normalization]\label{L-parti-norm}
Let $\mathcal D$ be an  essential, regular partition of $\Sigma$. Then, one can construct a normal partition $\widetilde{\mathcal D}$ of $\Sigma$ such that $\delta (\widetilde{\mathcal D},\Sigma) = \delta (\mathcal D,\Sigma)$, and $\omega(\widetilde{\cD},\Sigma) = \omega(\cD,\Sigma)$.
\end{lemma}

\begin{proof}
Using condition \eqref{E-parti-ess}, we see that an  essential, regular partition $\cD = \set{D_j}_{j=1}^k$ is normal except possibly at points in $\cS(\cD)$, with index $\iota > 1$, and for which there exists some domain $D_j$ such that $B(x,\varepsilon) \cap D_j$ has at least two connected components. Here, $B(x,\varepsilon)$ denotes the disk with center $x$ and radius $\varepsilon$ in $\Sigma$. Let $x$ be such a point. For $\varepsilon$ small enough, introduce the partition $\cD_x$ whose elements are the $D_j\sm \overline{B(x,\epsilon)}$, and the extra domain $B(x,\varepsilon)\cap \Sigma$. In this procedure, we have $\kappa(\cD_x) = \kappa(\cD)+1$; an interior singular point $x$, with $\nu(x) \ge 4$ is replaced by $\nu(x)$ singular points of index $3$ for which condition \eqref{E-parti-nor} is satisfied. Hence, $\sigma (\cD_x)= \sigma (\mathcal D) +1$ and $\delta (\cD_x)= \delta (\mathcal D)$. A similar procedure is applied at a boundary singular point. By recursion, we can in this way eliminate all the singular points at which condition \eqref{E-parti-nor} is not satisfied. Choosing $\varepsilon$ small enough, this procedure does not change the orientability of the modified domains, while the added disks are orientable.
\end{proof}

\begin{remark}\label{R-parti-norm}
By the same process, one could also remove all singular points with index $\iota > 1$. We do not need to do that for our purposes.
\end{remark}%

\subsection{Partitions and Euler-type formulas}

In the case of partitions of the sphere $\bS^2$, or of a planar domain $\Omega$, we have the following Euler-type formula, which appears in \cite{Ley1993,Ley1996} (sphere) or \cite{HOMiNa} (planar domain).


\begin{proposition}\label{P-euler-or}
Let $\Sigma$ be $\bS^2$, or a bounded open set $\Omega$ in $\mathbb R^2$, with piecewise $C^{1}$ boundary, and let $\mathcal D$ be a regular partition with $\partial \mathcal D$  the boundary set. Then,
\begin{equation}\label{E-euler-or}
\kappa(\cD,\Sigma) = 1 + \beta(\cD,\Sigma) + \sigma(\cD,\Sigma)\,.
\end{equation}
\end{proposition}

\begin{rem}
{\clm
\begin{remark}
The formula seems to work in case of cracks. Take for  example the disk
and a  $1$-partition consisting of the disk minus a radius.
In the reference, it is assumed that $\nu (x) \geq 3$ which is natural
when considering minimal partitions.
\end{remark}

This formula has been  applied, together with other arguments, to determine upper bounds for the number of singular points of minimal partitions.
}
\end{rem}

\begin{theorem}\label{T-euler-nor}
Let $\cD$ be an  essential, regular partition of the M\"{o}bius strip $M_1$.
With the previous notation, we have,
\begin{equation}\label{E-euler-nor}
\kappa(\cD,M_1) = \omega(\cD,M_1) + \beta(\cD,M_1) + \sigma(\cD,M_1)\,.
\end{equation}
\end{theorem}

\begin{rem}
{\clm One could dream, and try to prove,
\begin{theorem}\label{T-euler-nor-bis}
Let $\cD$ be a regular partition of the surface $\Sigma$, where $\Sigma$ is the real projective plane $\rP$, the Klein bottle $\K^2$, or of the M\"{o}bius strip $M_1$. With the previous notation, we have,
\begin{equation}\label{E-euler-nor-bis}
\kappa(\cD,\Sigma) = \omega(\cD,\Sigma) + \beta(\cD,\Sigma) + \sigma(\cD,\Sigma)\,.
\end{equation}
\end{theorem}
}
\end{rem}

\section{Proof of Theorem~\ref{T-euler-nor}}\label{S-pfnor}

 The idea to prove Theorem~\ref{T-euler-nor} is to examine how $\delta(\cD,\Sigma)$ changes when the partition $\cD$ and the surface $\Sigma$ are modified, starting out from a partition $\cD$ of the M\"{o}bius strip, and arriving at a partition $\widetilde{\cD}$ of a domain $\Omega$ in $\R^2$, on which we can apply \eqref{E-euler-or}.

\begin{lemma}\label{L-pfnor-2}
Let $\mathcal D$ be a regular partition of $M_1$. Assume that there exists a simple  piecewise $C^1$ curve $\ell : [0,1] \to M_1$, such that $\ell \subset \partial \cD$, $\ell(0), \ell(1) \in \partial M_1$,  and $\ell$ transversal to $\partial M_1$. If $\Omega:=M_1 \setminus \ell$ is simply-connected \footnote{More precisely, after scissoring $M_1$ along $\ell$, and  unfolding, one can view $\Omega$ as a subset $\widetilde \Omega$ of $\mathbb R^2$. The former $\ell$ is now split into two lines  $\ell_+$ and $\ell_-$ in the boundary of $\widetilde \Omega$. For simplicity, we denote by $\ell$, both the curve, and the image $\ell([0,1])$.}, then $\delta(\mathcal D,M_1)=0$.
\end{lemma}

\begin{proof}
Observing that $\mathcal D$ can also be viewed as a partition of $\Omega$, it is enough to prove that
$$\delta(\mathcal D,\Omega)= \delta(\mathcal D,M_1)\,.
$$
For this, we make the following observations.
\begin{itemize}
\item Since $1= b_0(\bD \cup \partial \Omega) = b_0(\bD \cup \partial M_1)$ and $b_0(\partial \Omega)= b_0(\partial M_1)$, we have
$$
 \beta(\mathcal D,\Omega) = \beta (\mathcal D,M_1)\,.
 $$
 \item  $\omega(\mathcal D, M_1)=  \omega(\mathcal D, \Omega)=0$. Indeed, each domain in $\cD$ is orientable because $\Omega$ is simply-connected.
 \item $\kappa(\mathcal D,M_1) = \kappa(\mathcal D,\Omega)$.
 \item Let $x$ be a singular point of  $\mathcal D$ in $M_1$ belonging to $\ell$, with index $\nu (x)$. After scissoring, we obtain two boundary points $x_{+} \in \ell_{+}$ and $x_{-} \in \ell_{-}$, with indices $\rho_+$ and $\rho_-$ such that $\rho_+ + \rho_- = \nu (x)- 2\,$.
 \item After scissoring, the boundary singular point $\ell(0)$ of $\ell$ yields two boundary singular points $y_{0,+}$ and $y_{0,-}$ such that $\rho_{0,+} + \rho_{0,-} =\rho(\ell(0)) -1$, with a similar property for $\ell(1)$.
 \item As a consequence of the two previous items, we have
 $$
 \sigma (\mathcal D,M_1)= \sigma (\mathcal D,\Omega) + 1\;.
 $$
\end{itemize}
Since $\Omega$ is homeomorphic to a simply-connected domain in $\R^2$, we have $\kappa(\cD,\Omega) = 1 + \beta(\cD,\Omega) + \sigma(\cD,\Omega)$. Taking into account the preceding identities, it follows that $\delta(\cD,\Omega) = 0$.
\end{proof}

\begin{lemma}\label{L-pfnor-4}
Let $\mathcal D$  be a regular  partition of $M_1$. Then, there exists a simple  piecewise $C^1$ path $\ell : [0,1] \to M$ such that:
\begin{itemize}
\item $M_1\setminus \ell$ is simply connected;
\item $\ell$ crosses $\bD$ and hits $\partial M_1$ transversally;
\item $\mathcal S(\cD)\cap \ell =\emptyset.$
\end{itemize}
\end{lemma}

\begin{proof}
Starting from any line $\ell_0$ such that $M_1\setminus \ell_0$ is simply connected, it is easy to deform $\ell_0$ in order to get the two other properties.
\end{proof}


\begin{lemma}\label{L-pfnor-6}
Let $\mathcal D$  be a regular  partition of $M_1$ whose elements $D_i$ are all simply-connected. Let $\ell: [0,1] \to M_1$ be a path as given by Lemma~\ref{L-pfnor-4}. Let $\widetilde{\mathcal D}$ be the partition of $M_1$ whole elements are the connected components of the $D_i \sm \ell$. Then, $\partial \widetilde{\mathcal D} =\partial \mathcal D \cup \ell$, and $\delta(\mathcal D,M_1)= \delta(\widetilde{\mathcal D},M_1)$. In particular $\widetilde{\mathcal D}$ satisfies the assumptions of Lemma \ref{L-pfnor-2}, and
 $$ \delta(\mathcal D,M_1) =0\,.$$
\end{lemma}

\begin{proof}
Denote by $y_i, x_1,\dots , x_N,y_f$ the points of $\ell \cap \partial \mathcal D$,
where
\begin{itemize}
\item[$\diamond$] $y_i$ and $y_f$ belong to $\partial M_1$;
\item[$\diamond$] each open interval $J$ on the path $\ell$, delimited by two consecutive points in the sequence, is contained in some element $D_J$ of the partition, and the end points of $J$ belong to $\partial D_J$.
   \end{itemize}

Using the assumption that the elements of $\cD$ are simply-connected, and comparing $\widetilde{\cD}$ with $\cD$, we observe that:
\begin{itemize}
  \item[$\diamond$] $\kappa(\widetilde{\cD})=\kappa(\mathcal D) + N+1$;
  \item[$\diamond$] the partition $\widetilde{\cD}$ has two extra simple singular boundary points, and $N$ interior singular points with index $\iota = 4$, so that
$$ \sigma( \widetilde{\mathcal D})= \sigma(\mathcal D) + N+1\,;$$
  \item[$\diamond$] $\beta(\widetilde{\cD}) = \beta(\cD)$ and $\omega(\widetilde{\cD}) = \omega(\cD) = 0$.
\end{itemize}
The lemma follows.
\end{proof}

\begin{lemma}\label{L-pfnor-8}
Let $S$ be (the interior of) a surface with or without boundary. Let $D$ be an open subset of $S$, and $K \subset D$ a compact subset. Assume that $D\sm K$ is connected. Then, $S\sm K$ is connected.
\end{lemma}%

\begin{proof} It suffices to prove the following claim.\smallskip

\emph{Claim.}~ Given any $x_0, x_1 \in S\sm K$, there exists a path $\gamma : [0,1] \to S\sm K$ such that $\gamma(0) = x_0$ and $\gamma(1) = x_1$.\medskip

\noindent $\diamond$~  Since $S$ is connected, there exists a path $\ell : [0,1] \to S$, with $\ell(0) = x_0$ and $\ell(1) = x_1$. If $x_0$ and $x_1$ both belong to $D$, the claim is clear. Without loss of generality, we may now assume that $x_0 \not \in D$.\smallskip

\noindent $\diamond$~ Define the set $J:= \set{t \in [0,1] ~|~ \ell(t) \not \in K}$. Since $x_0 \not \in K$, the set $J$ in not empty, and we define $t_0 := \inf J$. Clearly, $\ell(t_0) \in K$, and since $K \subset D$, $\ell(t_0) \in D$. Since $D$ is open, there exists some $\varepsilon > 0$ such that $\ell ([t_0 - \varepsilon,t_0+\varepsilon]) \subset D$, and hence $x_{\varepsilon} = \ell(t_0-\varepsilon) \in D$. If $x_1 \in D$, then there exists a path $\gamma_{\varepsilon}$ from $x_{\varepsilon}$ to $x_1$, contained in  $D\sm K$. The path $\gamma_{\varepsilon} \cdot \ell|_{[0,t_0-\varepsilon]}$ is contained in $S\sm K$ and links $x_0$ and $x_1$. \smallskip

\noindent $\diamond$~ If $x_1 \not \in D$, we can consider the path $\ell^{-1}$ and apply the preceding argument. Define $t_1 = \sup J$. Then $t_1$ exists and $\ell(t_1) \in K$, and there exists $\delta > 0$ such that $x_{\delta} = \ell(t_1+\delta) \in D$. There exists a path $\gamma_{\varepsilon,\delta}$ linking $x_{\varepsilon}$ and $x_{\delta}$ in $D\sm K$. The path $\ell|_{[t1+\delta,1]}\cdot \gamma_{\varepsilon,\delta} \cdot \ell|_{[0,t_0-\varepsilon]}$ is contained in $S\sm K$ and links $x_0$ and $x_1$.

\end{proof}%

\begin{lemma}\label{L-pfnor-10}
Let $\cD$ be a normal partition of $M_1$, and let $D$ be an element of $\mathcal D$. Assume that $D$ is orientable. Then, $D$ is homeomorphic to a sphere with $q$ discs removed, and one can find $q$  piecewise $C^1$ cuts $\ell_j$  (i.e. disjoint simple curves joining two points of $\partial D$ and  hitting the boundary transversally) such that $D\setminus \cup \ell_j$ is simply connected.
\end{lemma}

\begin{proof}
Since $\cD$ is normal, the domain $D$ is an orientable surface with boundary. According to the classification theorem \cite[Chap.~6]{GaXu}, $D$ is homeomorphic to some $\Sigma_{0,g,q}$, a sphere with $g$ handles attached, and $q$ disks removed. We claim that $g=0$. Indeed, if $g\not = 0$, there exists $g$ disjoint simple closed curves whose union does not disconnect $D$ and hence, by Lemma~\ref{L-pfnor-8} does not disconnect $M_1$. Since $M_1$ has genus $1$, we must have $g\le 1$. If $g=1$, we have a simple closed curve $\gamma$ which disconnects $D$, and preserves orientation since $D$ is orientable. A simple closed curve which disconnects $M_1$ reverses orientation, a contradiction, see \cite{Kre2014}.  One can draw cuts on the model surface $\Sigma_{0,0,q}$ and pull them back to $D$.

\begin{claim}\label{C-pfnor-10}
One can find piecewise $C^1$ cuts
\end{claim}%
Indeed, we can start from a continuous cut $\gamma :[0,1] \to M_1$, with $\gamma(0,1)\subset D$, joining two points $\gamma(0) \in \Gamma_0$ and $\gamma(1) \in \Gamma_1$, two components of $\partial D$. In order to finish the proof, we need to approximate $\gamma$ by a piecewise $C^1$ path $\gamma_1$ which is transversal to $\partial D$ at $\gamma(0)$ and $\gamma(1)$. For this purpose, we choose $B(\gamma(0),2r_0)\cap D$,  $B(\gamma(1),2r_1)\cap D$, with $r_0, r_1$ small enough so that these sets do not intersect $\partial D \sm \Gamma_i$. We choose $t_0, t_1 \in (0,1)$ such that $t_0 \in B(\gamma(0),r_0)$, $\gamma(t_1)\in B(\gamma(1),r_1)$. We cover $\gamma([t_0,t_1])$ with small disks $B(\gamma(t),r) \subset D$, with $r$ small enough so that the disks $B(\gamma(t),r)$ do not meet $\partial D$. By compactness, we can extract a finite covering of $\gamma([t_0,t_1])$. We can now easily construct the desired path $\gamma_1$, close to $\gamma$ by using geodesics segments contained in the disks.~~$\checkmark$

\end{proof}

\begin{lemma}\label{L-pfnor-12}
Let $\mathcal D$  be a normal  partition of $M_1$ such that all the domains $D_i \in \cD$ are orientable. Then, there exists a new partition $\widetilde{\mathcal D}$ of $M_1$ all of whose domains are simply-connected, and such that
$ \delta(\mathcal D,M_1) = \delta(\widetilde{\mathcal D},M_1)$.\\
In particular,
$$ \delta(\mathcal D,M_1) =0\,.$$
\end{lemma}

\begin{proof}
This is essentially an application of Lemma~\ref{L-pfnor-10}, with Lemma~\ref{L-parti-norm} in mind.  Each time, we apply one step in the proof of Lemma \ref{L-pfnor-10}  to one of the domains $D_i \in \cD$, $\beta$ decreases by $1$, and we create  two critical points $x_1$ and $x_2$ with $\iota (x_1) =\iota(x_2)=1$.
\end{proof}

\begin{lemma}\label{L-pfnor-20}
Let $\cD$ be a normal partition of $M_1$. Assume that some domain $D \in \cD$ is non-orientable. Then, there exists a connected component $\Gamma$ of $\partial D$, and a  piecewise $C^1$ path $\ell : [0,1] \to M_1$, such that $\ell(0), \ell(1) \in \Gamma$, with $\ell$ transversal to $\Gamma$ at its end points, $\ell((0,1)) \subset D$, and $D\setminus \ell$ is connected and orientable.
\end{lemma}

\begin{proof}
The partition $\cD$ being normal, the domain $D$ is a non-orientable surface with boundary and hence, there exists a homeomorphism $f : D \to \Sigma_{1,c,q}$,  one of the standard non-orientable surfaces, a sphere with $c \ge 1$ cross-caps attached, and $q$ discs removed ($c$ is the  genus\footnote{ For an orientable surface without boundary, the genus is defined as the number of handles attached to the sphere. For a non-orientable surface without boundary, the genus is defined as the number of cross-caps attached to the sphere.} of $\Sigma_{1,c}$,  and $q$ is the number of boundary components of $D$). The surface $\Sigma_{1,c,q}$ contains $c$ pairwise disjoint simple closed curves whose union does not disconnect the surface, each cross-cap contributes for one such curve. Since the genus of $M_1$ is $1$ ($M_1$ is a sphere with one cross-cap attached, and one disk removed), Lemma~\ref{L-pfnor-8} implies that $c=1$. It follows that $f: D \to \Sigma_{1,1,q}$, a M\"{o}bius strip $\Sigma_{1,1}$ with $q$ disks removed. Denote by $\Gamma = f^{-1}(\partial \Sigma_{1,1})$, the inverse image of the boundary of the M\"{o}bius strip. It is easy to cut $\Sigma_{1,1,q}$ by some path $\ell_0$, with end points on $\partial \Sigma_{1,1}$, in such a way that $\Sigma_{1,1,q} \sm \ell_0$ is simply-connected. The path $\ell_1$ is given by $\ell_1 = f^{-1}(\ell_0)$.   In order to finish the proof, it suffices to approximate $\ell_1$ by a piecewise $C^1$ path $\ell$ transversal to $\partial D$. For this purpose, we can use the same arguments as in the proof of Claim~\ref{C-pfnor-10}.
\end{proof}\medskip

\textbf{Proof of Theorem~\ref{T-euler-nor}}\\
 Let $\cD$ be a regular partition of the M\"{o}bius strip. Let $\delta(\cD,M_1)$ be defined in \eqref{E-parti-del}. By Lemma~\ref{L-parti-norm}, we may assume that $\cD$ is a normal partition, without changing the value of $\delta(\cD,M_1)$. \smallskip

\noid Assume that all the domains in $\cD$ are orientable. Applying Lemma~\ref{L-pfnor-12}, we conclude that $\delta(\cD,M_1) = 0$, and the theorem is proved in this case.\smallskip

\noid Assume that (at least) one of the domains in $\cD$, call it $D_1$ is non-orientable. We claim that $D_1$ is actually the only non-orientable domain in $\cD$. Indeed, assume that there is another non-orientable domain $D_2$. By Lemma~\ref{L-parti-norm}, both domains are surfaces with boundary, with genus $1$. Each $D_i, i=1,2$, contains a simple closed curve $\gamma_i$ which does not disconnect $D_i$. Since $D_1 \cap D_2 = \emptyset$, we would have two disjoint simple closed curves $\gamma_1, \gamma_2$. Applying Lemma~\ref{L-pfnor-8} twice, we obtain that $\gamma_1 \cup \gamma_2$ does not disconnect $M_1$.  This is a contradiction since $M_1$ has genus $1$. \smallskip

Apply Lemma~\ref{L-pfnor-20}: $D_1$ is homeomorphic to a M\"{o}bius trip with $q$ disks removed, there is a path $\ell$ which does not disconnect $D_1$, and such that $D_1\sm \ell$ is orientable. In doing so,  using the fact that $\ell$ is piecewise $C^1$ and transversal to $\Gamma$, we obtain a regular partition $\cD' = \set{D_1\sm\ell, D_2,...,D_{\kappa(\cD)}}$ of $M_1$. Since all the domains of $\cD'$ are orientable, by the preceding argument, we have $\delta(\cD',M_1) =0$. On the other-hand, we have $\omega(\cD',M_1) = 0$, $\beta(\cD',M_1) = \beta(\cD,M_1)$, $\kappa(\cD',M_1) = \kappa(\cD,M_1)$, and since we have the extra arc $\ell$ in $\bD'$, whose end points are singular points of index $\iota = 1$, we have $\sigma(\cD',M_1) = \sigma(\cD,M_1) +1$. It follows that,
\begin{equation*}
0 = \beta(\cD',M_1) + \sigma(\cD',M_1) - \kappa(\cD',M_1) = \beta(\cD,M_1) + \sigma(\cD,M_1) +1 - \kappa(\cD,M_1)\,,
\end{equation*}
and hence $\delta(\cD,M_1) = 0$.  The proof of Theorem~\ref{T-euler-nor} is now complete.\hfill \qed\
\vspace{1cm}

Figure~\ref{F-em-1} displays the typical nodal patterns of the Dirichlet eigenfunction
\begin{equation}\label{E-phi}
\Phi_{\beta,\theta}(x,y) = \cos\theta \sin(2x)\sin(3y) + \sin\theta \sin(3x)\sin(2y+\beta)\,,
\end{equation}
when $\beta \in (0,\pic)$ is fixed and $\theta \in (0,\pib)$. As explained in \cite[Section~5.4]{BHK}, there is a dramatic change in the nodal pattern when $\theta$ passes some value $\theta(\beta)$ (this value is precisely defined in \cite[Eq. (5.26)]{BHK}). For $0 <\theta \le \theta(\beta)$ the nodal domains are all orientable; for $\theta(\beta) < \theta < \pib$, there is one non-orientable nodal domain, homeomorphic to a M\"{o}bius strip (the domain in green). When $\theta = \theta(\beta)$, the nodal domain in green is not a surface with boundary due to the singular point at the boundary.\medskip

\begin{figure}[!ht]
  \centering
  \includegraphics[scale=0.35]{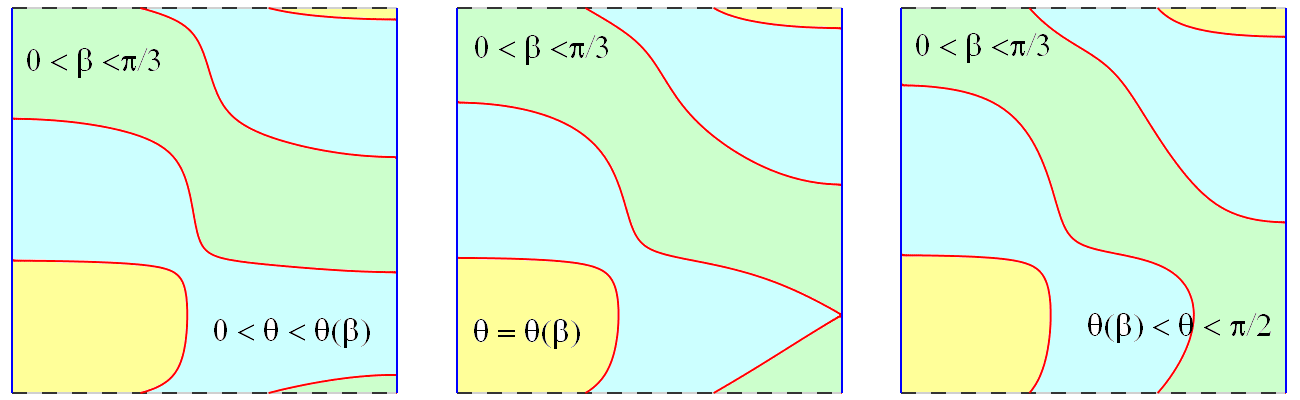}\\
  \caption{Example~1}\label{F-em-1}
\end{figure}

Figure~\ref{F-em-2} displays the nodal patterns of the eigenfunctions $\sin(3x)$ (left) and $\sin(5x)$ (right), on a 3D representation\footnote{We work with the flat metric on the M\"{o}bius strip, and use an embedding into $\R^3$ which is not isometric.} of the M\"{o}bius strip. The nodal domains are colored according to sign. In both cases, there is one nodal domain which is homeomorphic to a M\"{o}bius strip. The other nodal domains are cylinders, and hence orientable though not simply-connected.\medskip

\begin{figure}[h]
  \centering
  \includegraphics[scale=0.35]{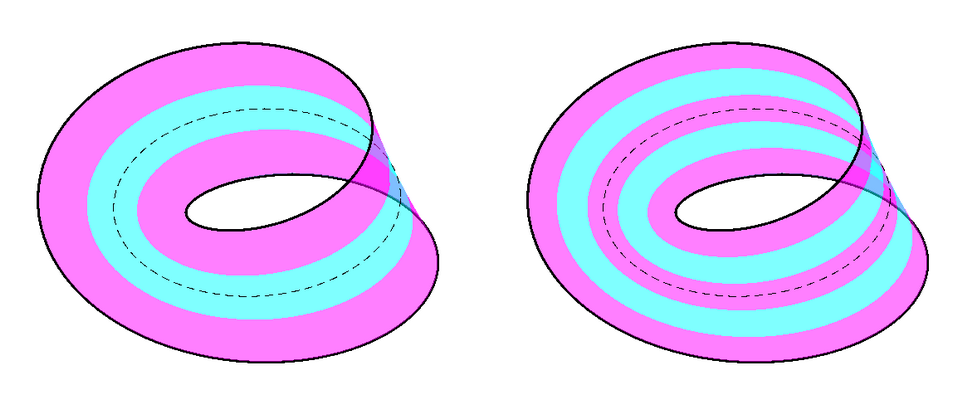}\\
  \caption{Example~2}\label{F-em-2}
\end{figure}

Figure~\ref{F-em-3}\,(A) displays the nodal patterns of the eigenfunction \eqref{E-phi}, with $\beta=\pic$ and $\pid < \theta \pib$. This is explained in \cite[Section~5.5]{BHK}. Notice that the nodal domains labeled (2) and (3) are not surfaces with boundary due to one of the singular points. There are two interior singular points (with $\nu=4$), and four boundary singular points (with $\rho=1$).
\smallskip

Figure~\ref{F-em-3}\,(B) displays the nodal pattern of the function
\begin{equation*}
\cos\theta \sin(x) \cos(6y) + \sin\theta \sin(6x) \cos(y),
\end{equation*}
with $\theta = 0.4\,\pi$. There is one non-orientable domain (colored in pink), homeomorphic to a M\"{o}bius strip with two holes (nodal domains colored in blue or in green). There is another disk-like nodal domain (colored in yellow).

\begin{figure}[h]
\centering
\begin{subfigure}[t]{.35\textwidth}
\centering
\includegraphics[width=\linewidth]{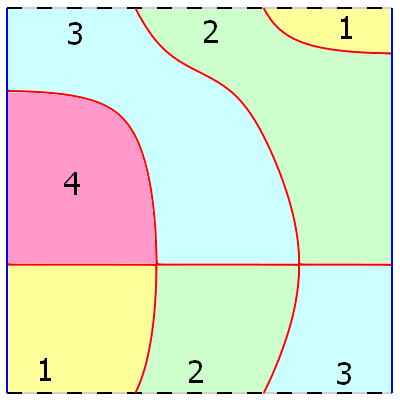}
\caption{}
\end{subfigure}
\begin{subfigure}[t]{.35\textwidth}
\centering
\includegraphics[width=\linewidth]{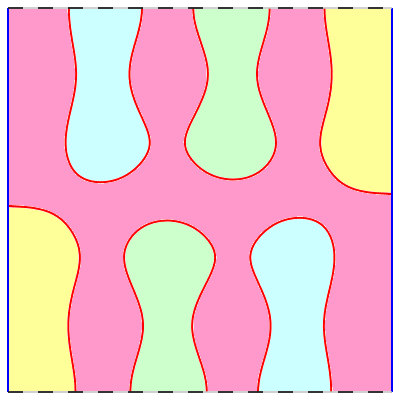}
\caption{}
\end{subfigure}
\caption{Example~3}\label{F-em-3}
\end{figure}

\begin{cylinder} 
\section{Euler-type formula for the M\"{o}bius strip via the covering cylinder}\label{S-cyl}

In this section, given a (nodal) partition $\cD = \set{D_j}_{j=1}^k$ of $M_1$, we aim at computing $\delta(\cD,M_1)$ by computing $\delta(\cD^{*},C_1)$, where $C_1$ is the cylinder which covers $M_1$, and $\cD^{*}$ the inverse image $\pi^{-1}(\cD)$, where $\pi$ is the projection map from $ C_1$ onto  $M_1$.\medskip

{\clr I am thinking in terms of nodal partitions. One could use a normal partition.\medskip}

According to \cite[Theorem~5.3.29, p.~164]{BeGo1988}, we have the following characterization of orientability in the present situation.

\begin{proposition}\label{P-cyl-2}
A nodal domain $D \in \cD$ is orientable if and only if $\pi^{-1}(D)$ has two connected components, and non-orientable if and only if $\pi^{-1}(D)$ is connected.
\end{proposition}

Let $(x_0,y_0)$ and $(\pi - x_0,y_0+\pi)$ be two points in $\pi^{-1}(D)$ which correspond to some point in $D$. Without loss of generality, we may assume that $y_0 = 0$. Since $\pi^{-1}(D)$ is connected, we can join these points by a simple path contained in $\pi^{-1}(D)$. By projection, we obtain a closed loop at $(x_0,0)$, and we can modify it, if necessary, so that it is simple. \medskip

\begin{lemma} \label{L-cyl-2} The partition $\cD$ contains at most one non-orientable nodal domain.
\end{lemma}

\begin{proof}
Indeed, assume that there are two such nodal domains $D_0$, $D_1$. Construct two paths as above, one from $(x_0,0)$ to $(\pi - x_0,\pi)$ in $\pi^{-1}(D_0)$, another from $(x_1,y_1)$ to $(\pi - x_1,y_1+\pi)$ in $\pi^{-1}(D_1)$.  It is easy to see that the projections of these paths must meet in $M_1$, which contradicts the fact that $D_0$ and $D_1$ are disjoint.
\end{proof}

It now remains to look at $\cD^{* }:=\pi^{-1}(\cD)$ in the annulus $C_1$ (where we can apply the Euclidean Euler formula, Proposition~\ref{P-euler-or}), and to consider the two possible cases:
\begin{itemize}
\item[(i)] $\cD$ contains only orientable domains,
\item[] and
\item [(ii)] $\cD$ contains precisely one non-orientable domain.
\end{itemize}\medskip

\emph{Case (i).} Let's look at $\cD^{*}$.  If there is a simple path in $\bD$ linking $\set{x=0}$ and $\set{x=\pi}$, then we can immediately conclude that $\omega(\cD,M_1)=0$, and $\delta(\cD,M_1) =0$, and we are done (see Lemma~\ref{L-pfnor-2}). \smallskip

Assume that there is no such path. Since we are in a Euclidean framework, with the obvious notation, we have
\begin{equation*}
0 = 1 + \beta^{*} + \sigma^{*} - \kappa^{*}.
\end{equation*}

From our assumptions, we have $\kappa^{*} = 2 \kappa$. Since $\pi$ is a covering map, we have $\sigma^{*} = 2 \sigma$. We observe that $\sigma$ is an integer. Indeed, we have the relation \cite[Eq. (2.11)]{BH0},
\begin{equation*}
\chi(M_1) + \sigma(\cD,M_1) = \sum_{j=1}^k \chi(\overline{D_j})\,.
\end{equation*}

Define $\beta_i$ as the number of connected components of $\bD$ which are not $\bD$-connected to $\partial M_1$, and $\beta_i^{*}$ as the number of connected components of $\bD^*$ which are not $\bD^*$-connected to $\partial C_1$.\medskip

Clearly, $\beta = \beta_i$. Since all domains are oriented, we have $\beta_i^* = 2 \beta_i$. The two connected components of $\partial C_1$ might be $\bD^*$-connected or not. Accordingly, $\beta^* = \beta_i^* - 1$ or $\beta_i^*$. The above equations imply that $\beta^*$ is odd, and hence we must have $\beta^* = 2 \beta_i - 1$. Finally, we obtain that $0 = \beta + \sigma - \kappa$. Since $\omega = 0$, the Euler-type formula for $\cD^*$ can be written as,
\begin{equation*}
0 = 1 + 2\beta -1 + 2\sigma - 2\kappa \,.
\end{equation*}
Noting that $\omega(\cD,M_1)=0$)  in the case at hand, we conclude that
\begin{equation*}
0 = \beta + \sigma - \kappa.
\end{equation*}
\smallskip

\emph{Case (ii).} In this case, we have $\omega(\cD,M_1) = 1$, $\kappa^{*} = 2 \kappa -1$ (use Lemma~\ref{L-cyl-2}), $\sigma^{*} = 2\sigma$. \medskip

\begin{lemma}\label{L-cyl-12}
Assume that $\cD$ has one (and hence only one) non-orientable domain. Then, $\beta^{*} = 2\beta$
\end{lemma}%

Assuming the claim is true, the Euclidean Euler formula,
\begin{equation*}
0 = 1 + \beta^{*} + \sigma^{*} - \kappa^{*}\,,
\end{equation*}
yields,
\begin{equation*}
0 = 1 + 2\beta + 2\sigma - 2\kappa + 1,
\end{equation*}
i.e.
\begin{equation*}
0 = 1 + \beta + \sigma - \kappa.
\end{equation*}

\begin{proof}[Tentative proof of the Lemma]
To compare $\beta$ and $\beta^*$, we first look at the components of the boundary in $\partial \cD\cup \partial M_1$. There is of course at least one such component. Hence $\beta (\mathcal D)$ is the number of components of $\partial \cD$ which are not $\bD$-connected to $\partial M_1$.\smallskip

As in Lemma~\ref{L-parti-norm}, we can transform $\cD$ to a partition $\widetilde{\cD}_{\varepsilon}$ such that $\delta(\cD,M_1) = \delta(\widetilde{\cD}_{\varepsilon},M_1)$, with the domain $D$ becoming a non-orientable domain $D_{\varepsilon}$ which is a surface with boundary homeomorphic to a M\"{o}bius strip $\Sigma$, with $q$ disks removed. Define $\Gamma = \partial \Sigma$.\medskip

There are two cases depending on whether $\Gamma$ is $\bD$-connected with $\partial M_1$ or not.  \smallskip

Let $\beta'(\mathcal D)$ be the number of components of $\partial \cD$ which are distinct from $\Gamma$, and from the components intersecting $\partial D_1$.

At this stage, we have
$$
\beta(\mathcal D) = \beta'(\mathcal D) + q-1  \,,
$$
in the first case, and
$$
 \beta(\mathcal D) = \beta'(\mathcal D) + q\,,
$$
in the second case.\medskip

When going to the covering $\pi^{-1} (\Gamma)$ should consist of two components. If not, we would get a contradiction with the existence of $\pi^{-1} (\gamma)$ where $\gamma$ is a non trivial path in $D$. The number of the connected components  of $\pi^{-1} (\partial D \sm \Gamma)$ is two-times the cardinal of the components of $\partial D \sm \Gamma$.\medskip

We finally look at the components of $\mathcal D^*$,  we have $b_0^* =2$ and the number of components of $\partial \mathcal C_1 \cup \bD^*)$ is also $2$ (again due to the existence of $\pi^{-1} (\gamma)$).
\end{proof}\medskip
\end{cylinder}

\begin{projective}
\section{Euler-type formula for partitions of the real projective space $\rP$, or the M\"{o}bius strip $M_1$, whose singular set is empty}\label{S-mtop}~

\begin{lemma}\label{L-mtop-4}
Let $\cD$ a regular partition of $\rP$, with empty singular set, $\cS(\cD) = \emptyset$. Then, $\delta(\cD,\rP) = 0$. More precisely, the boundary set $\bD$ is a disjoint union of circles $\gamma_1 \bigsqcup \cdots \bigsqcup \gamma_{\ell}$. If one of the circles, say $\gamma_1$ does not separate $\rP$, then the others, if any, separate $\rP$, and $\omega(\cD,\rP)=0$. If they all separate $\rP$, then $\omega(\cD,\rP) = 1$.
\end{lemma}%

\pf We examine both cases. \smallskip

\emph{Case~1. Assume that $\gamma_1$ does not separate.} In this case, $\gamma_1$ lifts to a unique circle $\gammat_1$ in $\bS^2$, $\bS^2 \setminus \gammat_1$ has two connected components $\Omega_1, \Omega_2$ which are both homeomorphic to a disk. The circle $\gammat_1$ is invariant under the antipodal map $a$, and $\Omega_1, \Omega_2$ are exchanged by $a$. This means that $\rP \setminus \gamma_1$ has only one connected component which is homeomorphic to a disk. Since the $\gamma_j$, $2 \le j \le \ell$ are disjoint from $\gamma_1$, they are contained in this disk, and each one lifts to two disjoint circles exchanged by $a$. In this case, we reduce the computation to the Euclidean case: we have $I(\cD) = 0$, $\sigma(\cD) = 0$, $\beta(\cD) = \kappa(\cD) = \ell$, and $\delta(\cD,\rP)=0$.\medskip

\emph{Case~2. Assume that the $\gamma_j$'s all  separate.} Look at one of them, say $\gamma_1$. It lifts to a pair of disjoint circles $\gammat_{1,1}$ and
$\gammat_{1,2}$ in $\bS^2$ and we have,
\begin{equation*}
\bS^2 \setminus \left( \gammat_{1,1} \bigsqcup \gammat_{1,2}\right) = \Omega_1 \bigsqcup \Omega_2 \bigsqcup A,
\end{equation*}
where $\Omega_1$ and $\Omega_2$ are homeomorphic to disks with boundary respectively $\gammat_{1,1}$ and $\gammat_{1,2}$, $A$ is an annulus with boundary $\gammat_{1,1} \bigsqcup \gammat_{1,2}$. Furthermore, the annulus $A$ is invariant under the antipodal map $a$, and $a(\Omega_1) = \Omega_2$.
For $2 \le j \le \ell$, the circle $\gamma_j$ lifts to two disjoint circles $\gammat_{j,1}$ and $\gammat_{j,2}$ in $\bS^2$ which are exchanged by $a$. Assume that $p$ circles lift in $\Omega_1 \cup \Omega_2$, call them $\gamma_1, \ldots, \gamma_p$, and that $q$ circles lift in $A$, call them $\gamma_{p+1}, \ldots, \gamma_{p+q}$, with $p+q = \ell$. Looking back at $\rP$, we see that $\gamma_1$ separates $\rP$ into a (set homeomorphic to a) disk with boundary $\gamma_1$, containing $(p-1)$ circles $\gamma_2,\ldots, \gamma_p$, and a M\"{o}bius strip containing $q$ circles, $\gamma_{p+1},\ldots, \gamma_{p+q}$. Then, $I(\cD,\rP) = 1$; $\beta(\cD,\rP) = 1 + p + q$,
$\sigma(\cD,\rP) = 0$, and $\kappa(\cD,\rP) = 2 + p + q$. It follows that $\delta(\cD,\rP)=0$. The lemma is proved. \hfill \qed \medskip

\begin{remark}\label{R-mtop-2} At least when $\cS(\cD) = \emptyset$, we have either $I(\cD,\rP) = 0$, or $I(\cD,\rP) = 1$.
\end{remark}%

\begin{lemma}\label{L-mtop-8}
Let $\cD = \set{D_j}_{j=1}^k$ be a regular partition of $M_1$ such that $\cS{\cD} = \emptyset$. Then $\delta(\cD,M_1)=0$.
\end{lemma}%

\pf We can view $M_1$ as embedded in $\rP$. Indeed,we can for example consider the map $F: [-\pi,\pi]\times [0,2\pi] \to \R^3$, given by
\begin{equation*}
(\theta,\phi) \mapsto \left( \cos(\frac{\theta}{2}) \cos\phi, \cos(\frac{\theta}{2}) \sin\phi, \sin(\frac{\theta}{2})\right).
\end{equation*}
Then, $F(-\theta,\phi+\pi) = - F(\theta,\phi)$, and the M\"{o}bius strip is homeomorphic to  the image $F\left( (-\frac{\pi}{2},\frac{\pi}{2})\times [0,\pi] \right)$. Then, $\rP$ is $M_1$ with a disk $D_0$ attached along $\partial M_1$. We obtain a partition $\cD_0 = \set{D_j}_{j=0}^k$ of $\rP$. We clearly have $\kappa(\cD_0,\rP) = \kappa(\cD,M_1) + 1$, and $I(\cD_0,\rP) = I(\cD,M_1)$. Since the partition is regular, $\sigma(\cD_0,\rP) = \sigma(\cD,M_1) = 0$. Finally $\beta(\cD_0,\rP) = \beta(\cD,M_1) + 1$. Applying Lemma~\ref{L-mtop-4}, we conclude that $\delta(\cD,M_1) = \delta(\cD_0,\rP) = 0$. \hfill\qed

\begin{remarks}\label{R-mtop-4}~\smallskip

(i) The way of counting the contribution of the singular points on the boundary, and in the interior also gives $\sigma(\cD_0,\rP) = \sigma(\cD,M_1)$ in the general case, and we should have $\delta(\cD,M_1) = 0$, provided we can prove the general formula for $\rP$.\smallskip

(ii) Even if we only work with nodal partitions, this method of proof requires that we can prove the formula of general partitions on $\rP$, indeed the partition $\cD_0$ is no longer a nodal partition.
\end{remarks}%
\end{projective}

\begin{append-A}
\appendix
\section{Complement of a circle in the real projective plane $\rP$}\label{S-mtop1}

Let $a : \bS^2 \to \bS^2$ be the antipodal map, for all $x \in \bS^2, a(x) = -x$. Let $\pi : \bS^2 \to \rP$ be the projection map from the sphere onto the real projective plane.

\begin{lemma}\label{L-mtop1-2}
Let $\Gamma$ be a circle in $\rP$, i.e., a simple closed curve, $\Gamma = \gamma([0,1])$, where $\gamma : \R \to \rP$ is $1$-periodic, continuous, and injective on $[0,1)$. Let $y_1, y_2$ be the inverse images of $x=\gamma(0) = \gamma(1)$ under the map $\pi$. For $i \in \set{1,2}$, let $\gammat_i$ be the lifting of the map $\gamma$, with $\gammat_i(0)=y_i$, and let $\Gammat_i = \gammat_i(\R)$. Then,
\begin{enumerate}
  \item either $\Gammat_1$ and $\Gammat_2$ coincide, i.e., $\pi^{-1}(\Gamma)$ is connected,
  \item or $\Gammat_1$ and $\Gammat_2$ are disjoint circles, i.e., $\pi^{-1}(\Gamma)$ has two connected components.
\end{enumerate}
In the first case, the circle $\Gammat_1 = \Gammat_2$ is invariant under the antipodal map. In the second case, the circles $\Gammat_1$ and $\Gammat_2$ are exchanged by the antipodal map.
\end{lemma}%

\pf Clearly, for $i \in \set{1,2}$, the curve $\gammat_i$ is injective on $[0,1)$, and $\gammat_i(1)$ is either $y_1$ or $y_2$. If $\gammat_i(1) \not = y_i$, we are in the first case. Indeed, the uniqueness of the lifting implies that the curves are identical up to a translation of the parameter by $1$. If $\gammat_i(1) = y_i$, we are in the second case, and the injectivity of $\gammat_i$ on $[0,1)$ implies that $\Gammat_1$ and $\Gammat_2$ are disjoint.  The last assertion follows from the uniqueness of the lifting, and the fact that $a(y_1)=y_2$ and vice-versa.\hfill \qed

\begin{lemma}\label{L-mtop1-4}
Let $\Gamma = \gamma([0,1]) \subset \rP$ be a simple closed curve. Assume that $\pi^{-1}(\Gamma) = \Gammat_1 \bigsqcup \Gammat_2$ is the union of two disjoint simple closed curves in $\bS^2$. Then, $\rP \sm \Gamma$ has two connected components. One component is homeomorphic to a disk, the other is homeomorphic to a M\"{o}bius strip, and they have $\Gamma$ as common boundary.
\end{lemma}%

\proof The assumption of the lemma means that we are in Case~(2) of Lemma~\ref{L-mtop1-2}. We use the same notation as in the proof of this lemma.\medskip

Since $\Gammat_1$ is a simple closed curve in $\bS^2$, the Jordan-Sch\"{o}nflies theorem implies that $\bS^2 \sm \Gammat_1$ has two connected components with common boundary, both homeomorphic to a disk (with homeomorphisms extending to the boundaries). Since $\Gammat_1$ and $\Gammat_2$ are disjoint, $\Gammat_2$ is entirely contained in one of these components. Call $\Omega_1$ the component which does not contain $\Gammat_2$, and $\Omega_1'$ the component which contains $\Gammat_2$. Since $\Omega_1'$ is a disk, we can again apply the Jordan-Sch\"{o}nflies theorem. The curve $\Gammat_2$ divides $\Omega_1'$ into two connected components, one component $\Omega_2$ homeomorphic to a disk, the other $A$, homeomorphic to an annulus, and whose boundary has two connected components, $\Gammat_1 = \partial \Omega_1$, and $\Gammat_2 = \partial \Omega_2$.\medskip

\begin{claim}\label{Cl-mtop1-2}
With the previous notation,
\begin{enumerate}[(i)]
  \item $a(\Omega_1) = \Omega_2$,
  \item $a(A) = A$.
\end{enumerate}
\end{claim}%

\emph{Proof of Claim~\ref{Cl-mtop1-2}}. \smallskip

\noid \emph{Subclaim~1.}~ Either $a(\Omega_1) \subset \bS^2\sm \overline{\Omega_2}$ or $a(\Omega_1) \subset \Omega_2$. Indeed, there would otherwise exist $x, y \in \Omega_1$ such that $a(x) \in \bS^2 \sm \Omega_2$ and $a(y) \in \Omega_2$. One can choose a curve $c:[0,1] \to \Omega_1$ such that $c(0) = x$ and $c(1) = y$. The curve $t \mapsto a(c(t))$ going from $a(x)$ to $a(y)$ must cross the boundary: there exists some $t_0 \in (0,1)$ such that $a(c(t_0)) \in \Gammat_2$ and hence $c(t_0) \in \Gammat_1$, a contradiction since $c$ is entirely contained in $\Gammat_1$. \hfill $\checkmark$ \smallskip

\noid \emph{Subclaim~2.}~ Either $a(\Omega_2) \subset \bS^2\sm \Omega_1$ or $a(\Omega_2) \subset \Omega_1$. Analogous to Subclaim~1.\smallskip

\noid \emph{Subclaim~3.}~ Assume that $a(\Omega_1) \subset \bS^2\sm \Omega_2$. Then, $a(\Omega_1) \subset A$. Indeed, by applying Brouwer's theorem to the disk $\Omega_1$, we see that $a(\Omega_1) \not \subset \overline{\Omega_1}$. Since $a(\Omega_1) \subset \bS^2\sm \Omega_2$, another possibility would be that there exist $x,y \in \Omega_1$ such that $a(x) \in \Omega_1$ and $a(y) \in A$. We could again choose a path from $x$ to $y$ in $\Omega_1$, and conclude that there exists some $t_0 \in [0,1]$ such that $a(c(t_0)) \in \Gammat_1$, so that $c(t_0) \in \Gammat_2$, a contradiction. The remaining possibility is $a(\Omega_1) \subset A$, as claimed.\hfill $\checkmark$ \smallskip

\noid \emph{Subclaim~4.}~ Assume that $a(\Omega_2) \subset \bS^2\sm \Omega_1$. Then, $a(\Omega_2) \subset A$. Analogous to Subclaim~3.\smallskip

\emph{Proof of Claim~\ref{Cl-mtop1-2}, Assertion~(i).} Assume that $a(\Omega_1) \subset \bS^2 \sm \Omega_2$. Then, by Subclaim~3, $a(\Omega_1) \subset A$. Using Subclaims~2 and 4, we see that either $a(\Omega_2) \subset A$, or $a(\Omega_2) \subset \Omega_1$. If $a(\Omega_2) \subset \Omega_1$, then $\Omega_2 \subset a(\Omega_1) \subset A$ which is absurd. It follows that we must have $a(\Omega_1) \subset A$ and $a(\Omega_2) \subset A$. this implies that $\Omega_1, \Omega_2 \subset a(A)$. Since $\bS^2 = \overline{\Omega_1 \cup \Omega_2 \cup A}$, we would have $a(A) = \Omega_1 \cup \Omega_2$, contradicting the fact that $A$ is connected. Considering the various possible cases, we can conclude that $a(\Omega_1) \subset \Omega_2$ and $a(\Omega_2) \subset \Omega_1$, and hence $a(\Omega_1) = \Omega_2$. \smallskip

 \emph{Proof of Claim~\ref{Cl-mtop1-2}, Assertion~(ii).} Follows immediately from Assertion~(i). \hfill $\checkmark$ \medskip

\emph{End of the proof of Lemma~\ref{L-mtop1-4}.}~ From Claim~\ref{Cl-mtop1-2}, we deduce that
$$
\rP = \pi(\Omega_1) \bigsqcup \pi(A) \bigsqcup \Gamma,
$$
and the lemma follows. \hfill \qed

\begin{lemma}\label{L-mtop1-6}
Let $\Gamma = \gamma([0,1]) \subset \rP$ be a simple closed curve. Assume that $\pi^{-1}(\Gamma) = \Gammat$ is a simple closed curve in $\bS^2$. Then, $\rP \setminus \Gamma$ has one connected component, homeo\-morphic to a disk, with boundary $\Gamma$.
\end{lemma}%

\pf The assumption means that we are in Case~(1) of Lemma~\ref{L-mtop1-2}. The curve $\Gamma$ lifts to a single simple closed curve $\Gammat$ whose complement in $\bS^2$ has two connected components, $\Omega_1$ and $\Omega_2$, both homeomorphic to the disk, and sharing the same boundary $\Gammat$. The only possibilities are $a(\Omega_i) = \Omega_i$, for $i \in \set{1,2}$, or $a(\Omega_1) = \Omega_2$. The first possibility can be discarded by using Brouwer's theorem, as above. Then, $\rP \sm \Gamma = \pi(\Omega_1)$ is homeomorphic to a disk. \hfill \qed

\begin{remark}\label{R-mtop1-2}
The above statements appear in \cite[p.~286]{BCR1998}, see also \cite{SG1}.
\end{remark}%

\end{append-A}

\vspace{2cm}
\bibliographystyle{plain}

\begin{thebibliography}{1}

\bibitem{BH0} P. B\'erard and B. Helffer.
\newblock Remarks on the boundary set of spectral equipartitions.
\newblock  Philosophical Transactions of the Royal Society A 2014 372, 20120492, published 16 December 2013.




\bibitem{BHK} P.~B\'{e}rard, B.~Helffer and R.~Kiwan.
\newblock Courant-sharp property for Dirichlet eigenfunctions  on the M\"{o}bius strip.
\newblock arXiv:2005.01175.


\bibitem{BeGo1988} M.~Berger and B.~Gostiaux.
\newblock Differential geometry: manifolds, curves and surfaces.
\newblock Springer, 1988.


\begin{append-A}
\bibitem{BCR1998} J.~Bochnak, M.~Coste, and MF. Roy.
\newblock Real algebraic geometry.
\newblock Springer, 1998.
\end{append-A}


\bibitem  {BNH1}  V. Bonnaillie-No\"el, B. Helffer.
\newblock Nodal and spectral minimal partitions --The state of the art in
2016--.
\newblock In Shape Optimization and Spectral Theory (pp. 353-397)
(A. Henrot Editor).
\newblock De Gruyter Open.




\bibitem{GaXu} Gallier Jean, Xu Dianna (2013).
\newblock A guide to the classification theorem for compact surfaces,
\newblock  collection ``Geometry and Computing'', Springer, Heidelberg. \href{http://dx.doi.org/10.1007/978-3-642-34364-3}{doi: 10.1007/978-3-642-34364-3}









\bibitem{HHOT}
B. Helffer, T. Hoffmann-Ostenhof, and S. Terracini.
\newblock
Nodal domains and spectral minimal partitions.
\newblock {\it Ann. Inst. H. Poincar\'e  AN} 26, 101--138, 2009.




\bibitem{HOMiNa} T. Hoffmann-Ostenhof, P. Michor, and N. Nadirashvili.
\newblock Bounds on the multiplicity of eigenvalues for fixed membranes.
\newblock Geom. Funct. Anal. 9 (1999), no. 6, 1169--1188.






\bibitem{Kre2014} M.~Kreck
\newblock See the entries ``Orientation covering'' and ``Orientation of manifolds'' in the \href{http://www.boma.mpim-bonn.mpg.de/volumes/definitions}{Bulletin of Manifold Atlas.}

\bibitem{Ley1993} J.~Leydold.
\newblock Nodal properties of spherical harmonics.
\newblock PhD Thesis, Universit\"{a}t Wien, 1993.

\bibitem{Ley1996} J.~Leydold.
\newblock On the number of nodal domains of spherical harmonics.
\newblock Topology 35 (1996) 301--321.





\begin{append-A}
\bibitem{SG1} H.~P.~de~Saint-Gervais.
\newblock Cours moderne.
\newblock \href{http://analysis-situs.math.cnrs.fr/Topologie-des-courbes-algebriques-planes-reelles.html}{Topologie des courbes alg\'{e}briques planes r\'{e}elles.}
\end{append-A}


\end{thebibliography}

\end{document}